\documentclass[12pt]{article}
\usepackage{a4}
\usepackage{amsmath}
\usepackage{amssymb}
\usepackage{amsfonts}
\usepackage{amsthm}
\usepackage{graphicx}
  \newtheorem{thm}{Theorem}[section]
 
 \newtheorem{prop}[thm]{Proposition}

 \newtheorem{lemma}[thm]{Lemma}
 \newtheorem{rem}[thm]{Remark}
 
\DeclareMathOperator{\Log}{Log}

\def\C{\mathbb{C}}

\def\L{\mathbb{L}}
\def\N{\mathbb{N}}

\def\a{\alpha}
\def\b{\beta}

\author {Christian Berg and Henrik L.\ Pedersen}
\title {A family of Horn-Bernstein functions}

\date{\today}

\begin{document}
\maketitle

\begin{abstract}
 A family of recently investigated Bernstein functions is revisited and those functions for which the derivatives are logarithmically completely monotonic are identified. This leads to the definition of a class of Bernstein functions, which we propose to call Horn-Bernstein functions because of the results of Roger A. Horn. 
\end{abstract}


\section{Introduction and main results}

The family of functions
\begin{equation}\label{eq:ha}
h_\a(z)=(1+1/z)^{\a z}:=\exp(\a z\Log(1+1/z)),\quad \a\in \C, 
\end{equation}
defined for $z$ in the cut plane $\mathcal A:=\C\setminus]-\infty,0]$,
has been examined in various publications, latest in \cite{B:M:P}, which contains references to previous treatments. Note that $\Log:\mathcal A\to\C$ is the principal branch of the logarithm, holomorphic in $\mathcal A$ and real on the positive half-line.

The functions in \eqref{eq:ha} appeared in the paper \cite{B} with the goal of finding the set of exponents $\a>0$  such that $h_\a$ is a Bernstein function or equivalently such that % to find the set of $\a>0$ for which
%\begin{equation}\label{eq:fa}
$f_\a(x)=e^\a-h_\a(x)$
%\end{equation} 
is a completely monotonic function. This problem was inspired by a remark in \cite[p.458]{A:B}.

We  adopt the notation of \cite{S:S:V} and denote by $\mathcal{CM}$ the set of completely monotonic functions and  $\mathcal{BF}$ the set of Bernstein functions. See the monographs \cite{B:F} and \cite{S:S:V} for a treatment of these classes of functions.

A family $\{\varphi_\a\}_{\a\in\C}$ of entire functions was found in  \cite[Theorem 2.10]{B:M:P}   such that
\begin{equation}\label{eq:phia}
f_\a(z)=\int_0^\infty e^{-sz}\varphi_\a(s)\,ds,\quad \Re z>0.
\end{equation}
These functions, initially given by a contour integral, were shown to have the power series expansion
%\begin{equation}\label{eq:ent1}
$$
\varphi_\a(s)=e^\a\sum_{n=0}^\infty(-1)^np_{n+1}(\a)\frac{s^n}{n!},\quad \a, s\in\C,
$$
%\end{equation}
where $(p_n)_{n\ge 0}$ denotes the sequence of polynomials recursively defined by $p_0(\a)=1$ and 
%\begin{equation}\label{eq:poly-rec}
$$
p_{n+1}(\a)=\frac{\a}{n+1}\sum_{k=0}^n \frac{k+1}{k+2} p_{n-k}(\a),\quad n\ge 0.
$$
%\end{equation}
(Notice that
%\begin{equation}\label{eq:poly1} 
$p_1(\a)=\a/2$ and $p_2(\a)=\a/3+\a^2/8$.)
%\end{equation}

Because of \eqref{eq:phia} and the theorem of Bernstein, we conclude that $f_\a\in \mathcal{CM}$ 
%is completely monotonic 
if and only if $\varphi_{\a}(s)\ge 0$ for $s>0$. In \cite[Theorem 1.7]{B:M:P} it was numerically established that there exists a number
\begin{equation}\label{eq:a*}
 \a^*\approx 2.29965\, 64432\, 53461\, 30332 
\end{equation}
such that $\varphi_\a$ is non-negative on $]0,\infty[$ if and only if $0\le \a\le\a^*$. For $\a$ in this interval $\varphi_\a$ is integrable over $[0,\infty[$ and  \eqref{eq:phia} holds for $\Re z\ge 0$. Furthermore, cf. \cite[Theorem 2.11]{B:M:P},
the Bernstein representation of $h_\a$ is
\begin{equation}\label{eq:bernha}
h_\a(z)=1+\int_0^\infty(1-e^{-sz})\varphi_\a(s)\,ds,\quad \Re z\ge 0, \quad 0\le \a\le \a^*,
\end{equation}
so the  L\'evy measure of $h_\a$ has the density $\varphi_\a$ with respect to Lebesgue measure.

For $0<\a\le 1$ there exists an integral formula for $\varphi_\a$, namely
%\begin{equation}\label{eq:St1}
$$
\varphi_\a(s)=\frac{1}{\pi}\int_0^1 (x/(1-x))^{\a x}\sin(\a\pi x) e^{-sx}\,dx,\quad 0<\a<1,\quad s\ge 0,
$$
%\end{equation}
and
%\begin{equation}\label{eq:St1_1}
$$\varphi_1(s)=e^{-s}+\frac{1}{\pi}\int_0^1 \left(x/(1-x)\right)^{x}\sin(\pi x) e^{-sx}\,dx,\quad s\ge 0.
$$
%\end{equation}
This shows that $\varphi_\a$ is not only non-negative but in fact that $\varphi_\a\in \mathcal{CM}$.  
%completely monotonic on $]0,\infty[$. 
Proofs are given in \cite[Section 3]{B:M:P}.  

Formula \eqref{eq:bernha} shows that $h_\a$ is a complete Bernstein function (see \cite[Chapter 6]{S:S:V}) for $0<\a\le 1$ and clearly also for $\a=0$. On the other hand $f_\a$ is not a Stieltjes function (defined below) for $\a>1$ by \cite[Theorem 3.1]{B:M:P}, so using the notation from \cite{S:S:V} we have
$$%\begin{equation}\label{eq:cbf1}
h_\a\in\mathcal{CBF} \iff \a\in[0,1].
$$%\end{equation}

The present paper started as an attempt to find a more direct proof that $h_2\in \mathcal{BF}$, i.e., that $h_2'\in \mathcal{CM}$. Our idea was to examine if $h_2'$ is a so-called logarithmically completely monotonic function. This is a stronger statement  since the class $\mathcal L$ of logarithmically completely monotonic functions is defined and characterized by the following result:

\begin{thm}
\label{thm:horn}
 The following conditions for a $C^\infty$-function $f: (0,\infty)\to (0,\infty)$ are equivalent and characterize the class $\mathcal L$:
 \begin{enumerate}
  \item[(i)] $-(\log f)'=-f'/f$ is completely monotonic,
  \item[(ii)] $f^{c}$ is completely monotonic for all $c>0$,
  \item[(iii)] $f^{1/n}$ is completely monotonic for all $n=1,2,\ldots$.
 \end{enumerate}
\end{thm}
This result goes back to Horn \cite{H}. For more recent proofs and historical comments see \cite{HJ}, \cite{berg1}, and \cite{B2008}. Because of property (ii) we augment the class $\mathcal L$ to
$$%\begin{equation}\label{eq:l0}
\mathcal L_0:=\mathcal L\cup \{0\}.
$$%\end{equation}
We define a new class called Horn-Bernstein functions, and denoted by $\mathcal{HBF}$ as
%of Bernstein functions  
$$%\begin{equation}\label{eq:lbf}
\mathcal{HBF}:=\{f\in\mathcal{BF} \mid f'\in\mathcal L_0\}.
$$%\end{equation}
Our main result is the following:
\begin{thm}\label{thm:main} There exists a number $0\le \b^*<\a^*$ such that
$$%\begin{equation}\label{eq:beta*}
h_\a\in\mathcal{HBF} \iff 0\le\a\le\b^*.
$$%\end{equation}
We have $\b^*\approx 2.18858\,63446\,61757\,09765$. 
\end{thm}

For a given  $\lambda >0$ a function $f:(0,\infty)\to \mathbb R$ is called a generalized Stieltjes function of order $\lambda$ if 
\begin{equation}
\label{eq:def-S}
f(x)=\int_0^{\infty}\frac{d\mu(t)}{(x+t)^{\lambda}}+c,
\end{equation}
where $\mu$ is a positive measure on $[0,\infty)$ making the integral converge for $x>0$ and $c\geq 0$. The set of generalized Stieltjes functions of order $\lambda$ is denoted $\mathcal S_\lambda$. Note that a function in $\mathcal S_\lambda$ has a holomorphic extension to the cut plane $\mathcal A$.
For additional information on these classes see e.g.\ \cite{K:P2}.

The class $\mathcal S_1$ of generalized Stieltjes functions of order $1$ is just denoted $\mathcal S$, and its members are simply called Stieltjes functions.

We remark that $f$ is a generalized Stieltjes function of order $\lambda$ of the form \eqref{eq:def-S} if and only if 
$$%\begin{equation}\label{eq:sokalabsolutely}
f(x)=\frac{1}{\Gamma(\lambda)}\int_0^{\infty}e^{-xt}t^{\lambda-1}\kappa(t)\, dt+c, \quad x>0,
$$%\end{equation}
where $\kappa\in \mathcal{CM}$. In the affirmative case, 
$
\kappa(t)=\int_0^{\infty}e^{-ts}\, d\mu(s).
$
See \cite[Eqn.(3)]{B:F1} or  \cite[Lemma 2.1]{K:P2}. This characterization shows also that
$\mathcal S_{\lambda_1}\subset \mathcal S_{\lambda_2}$ for $\lambda_1<\lambda_2$.

In accordance with \cite[Chapter 8]{S:S:V} we define for $\lambda>0$ the classes
$$%\begin{equation}\label{eq:bfla}
\mathcal{TBF}_{\lambda}:=\{f\in\mathcal{BF} \mid f'\in\mathcal S_{\lambda}\}, 
$$%\end{equation}
which can also be characterized as the set of Bernstein functions for which the L\'evy measure has a density $m(s)$ such that $s^{2-\lambda}m(s)$ is completely monotonic.
Clearly
$$
\mathcal{TBF}_{\lambda_1}\subset\mathcal{TBF}_{\lambda_2}\;\mbox{for}\;
0<\lambda_1<\lambda_2.
$$
The class $\mathcal{TBF}_1$ is also known as the class of Thorin-Bernstein functions simply denoted $\mathcal{TBF}$, see \cite[Chapter 8]{S:S:V} and $\mathcal{TBF}_2=\mathcal{CBF}$ (which is easily verified) is the class of complete Bernstein functions. Since $\mathcal S_2\subset \mathcal L_0$ by \cite{B:K:P} we have
$$
\mathcal{TBF}_{\lambda}\subset\mathcal{CBF}\subset\mathcal{HBF},\quad 0<\lambda<2.
$$
Concerning the family $h_\a,\a\ge 0$ we can summarize the known and new results as follows:
\begin{enumerate}
\item[(i)]
For $0<\lambda<2$ we have $h_\a\in\mathcal{TBF}_{\lambda}$ if and only if $\a=0$.
\item[(ii)]
$h_\a\in \mathcal{CBF}$ if and only if $0\le\a\le 1$.
\item[(iii)]
$h_\a\in \mathcal{HBF}$ if and only if  $0\le\a\le \beta^*$.
\item[(iv)]
$h_\a\in \mathcal{BF}$ if and only if  $0\le\a\le \a^*$.
\end{enumerate}
 
Only $\rm{(i)}$ requires a comment.
If $h_\a\in\mathcal{TBF}_\lambda$ for $0<\lambda<2$, then necessarily $0\le\a\le 1$, and by Equation \eqref{eq:bernha} the L\'evy measure for $h_\a$ has the density $\varphi_\a$, which must satisfy $s^{2-\lambda}\varphi_\a(s)\in\mathcal{CM}$. Since $\varphi_\a(0)=e^\a(\a/2)$ is finite, this is only possible for $\a=0$.

\begin{rem}{\rm There is a one-to-one correspondence between the set $P([0,\infty[)$ of infinitely divisible probability measures $\sigma$ on $[0,\infty[$  and the set of  Bernstein functions $f$ satisfying $f(0)=0$ via Laplace transformation
$$
\L(\sigma)(x):=\int_0^\infty e^{-tx}\,d\sigma(t)=e^{-f(x)},\quad x\ge 0.
$$  

Define $T_\lambda\subset P([0,\infty[), \lambda>0$  by 
$$
T_\lambda:=\{\sigma\in P([0,\infty[) \mid \L(\sigma)=e^{-f}, f\in\mathcal{TBF}_\lambda, f(0)=0\}.
$$
In \cite{B:F1} there is a discussion of these sets and their relation to exponential families.
}
\end{rem}

\section{Preliminary results}
A computation shows that $h_\a'(x)=\a h_\a(x)\rho(x)$, where
\begin{equation}\label{eq:rho}
\rho(x)=\log(1+1/x)-\frac{1}{x+1}=\int_0^\infty e^{-tx}\left(\frac{1-e^{-t}}{t}-e^{-t}\right)dt.
\end{equation}
This function $\rho$ together with the function $g$ defined as
\begin{equation}
\label{eq:def-g}
g(x)=-\frac{\rho'(x)}{\rho(x)}
\end{equation}
will be important in our investigations since $h_\a' \in\mathcal L$ if and only if
\begin{equation}\label{eq:hainL}
-\frac{h_\a''}{h_\a'}=g-\a\rho\in \mathcal{CM}.
\end{equation}

\begin{rem}{\rm
We notice that $\rho\in \mathcal{S}_2\setminus \mathcal{S}$ since
$$
\int_0^1\frac{t}{(x+t)^2}\,dt=\rho(x)=\int_0^1\frac{1}{x+t}\,d(t- \delta_1(t)),
$$
$\delta_1$ denoting the Dirac probability measure with mass at the point $1$.}
\end{rem}

The following results are essential ingredients in the search of $\a$ for which \eqref{eq:hainL} holds.

\begin{prop}\label{thm:g}
The function $g$ in \eqref{eq:def-g} is a Stieltjes function:
\begin{equation}\label{eq:g}
g(x)=\frac{1}{x(x+1)[(x+1)\log(1+1/x)-1]}=\frac{1}{x+1}+\int_0^1\frac{\tau_0(t)}{x+t} dt,
\end{equation}
where $\tau_0$  is a probability density on $]0,1[$ given by
\begin{equation}\label{eq:tau}
\tau_0(t)=\left(t[(1-t)\log((1-t)/t)-1]^2 +\pi^2t(1-t)^2\right)^{-1}.
\end{equation}
The function $\tau_0$ is convex with a unique minimum at $t^*\approx 0.592$ with minimum value $m\approx 0.569$. It decreases strictly from $\infty$ to $m$ on $]0,t^*[$ and increases strictly from $m$ to 1 on $]t^*,1[$.  
\end{prop}

%Theorem~\ref{thm:main} is based on the following result:

%\begin{figure}
 \begin{center}
\includegraphics[scale=0.4]{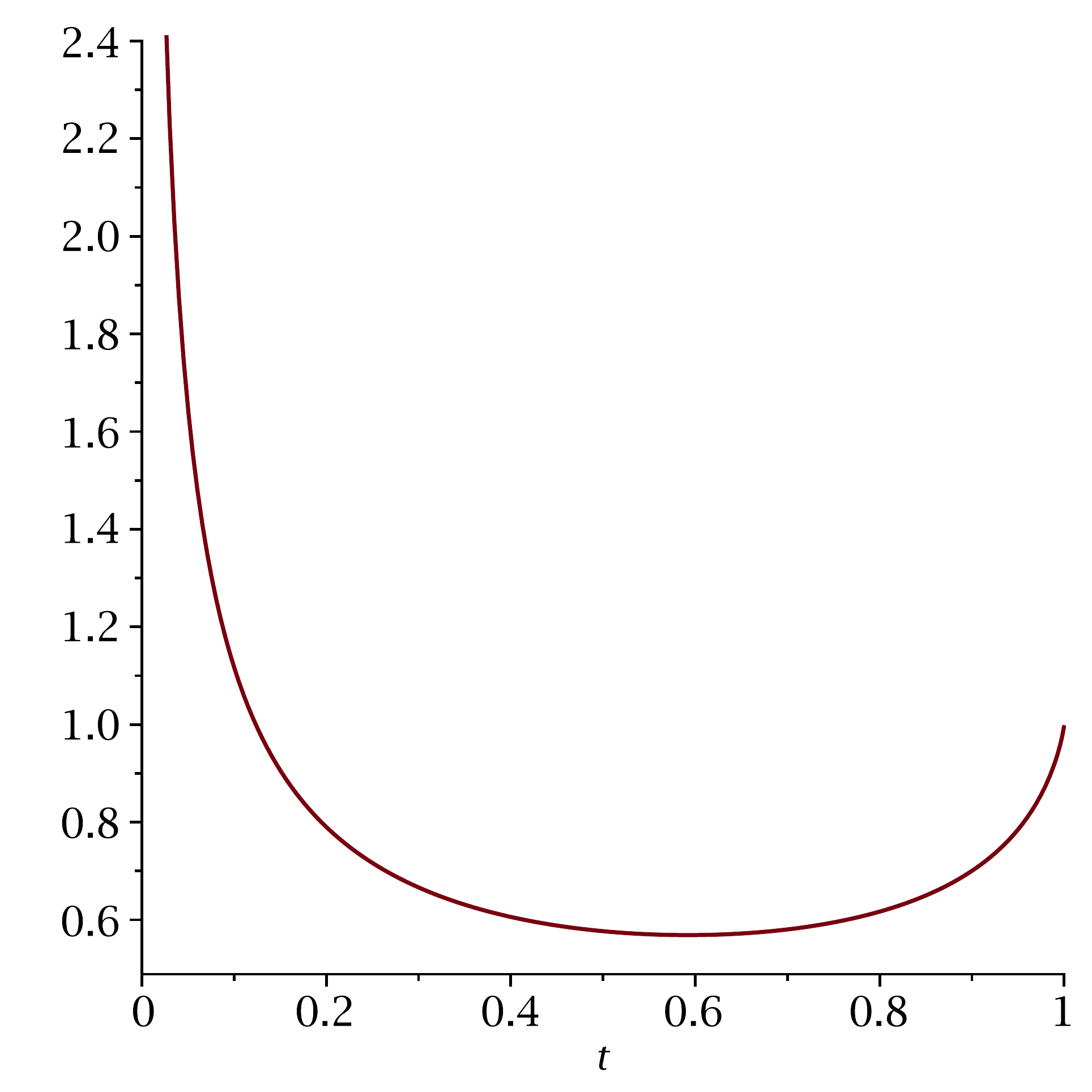}\\
 {The function $\tau_0$}
 \end{center}
%\end{figure}

\noindent {\it Proof of Proposition~\ref{thm:g}.} We start by noticing that $\Log(1+1/z)$ is holomorphic in $\C\setminus[-1,0]$  and that
$$
(z+1)\Log(1+1/z)-1=\int_0^1\frac{1-t}{z+t}\,dt,\quad z\in\C\setminus[-1,0].
$$
In particular $(z+1)\Log(1+1/z)-1\neq 0$ for $z\in\C\setminus[-1,0]$. Therefore $\rho(z)\neq 0$ and $g(z)$ defined in \eqref{eq:def-g} 
is holomorphic for $z\in\C\setminus[-1,0]$.

The first expression for $g$ in \eqref{eq:g} is easy, and to verify $g\in\mathcal S$ we use that it is equivalent to $1/(zg(z))\in\mathcal S$, cf. \cite[p. 25]{B2008}. However, 
\begin{eqnarray*}\label{eq:g1}
1/(zg(z))&=&(z+1)[(z+1)\Log(1+1/z)-1]=\int_0^1\frac{(z+1)(1-t)}{z+t}\,dt \\
&=&\int_0^1(1-t)\,dt +\int_0^1\frac{(1-t)^2}{z+t}\,dt=\frac12+\int_0^1\frac{(1-t)^2}{z+t}\,dt,
\end{eqnarray*}
showing that $1/(zg(z))\in\mathcal S$.
Hence, $g$ is a Stieltjes function, and holomorphic in $\C\setminus[-1,0]$. Therefore, 
$$
g(z)=c+\int_0^1\frac{d\tau(t)}{z+t}
$$
for a positive measure $\tau$ on $[0,1]$ and $c\ge 0$. Since $g(x)\to 0$ for $x\to\infty$, we have $c=0$. To find the measure $\tau$ we use the Stieltjes-Perron inversion formula in the form used in the proof in \cite[Lemma 1]{A:B}.

For $0<t<1, y>0$ we find
\begin{eqnarray*}
\lim_{y\to 0^+}g(-t+iy)&=&\left\{-t(1-t)[(1-t)(\log((1-t)/t)-i\pi)-1]\right\}^{-1}\\
&=&\frac{-1}{t(1-t)}\left\{(1-t)\log((1-t)/t)-1-i\pi(1-t)\right\}^{-1}
\end{eqnarray*}
uniformly for $t$ in compact subsets of $]0,1[$. Therefore,
$$
-\frac{1}{\pi}\lim_{y\to 0^+}\Im{g(-t+iy)}=\left\{t[(1-t)\log((1-t)/t)-1]^2+\pi^2 t(1-t)^2\right\}^{-1},
$$
which shows that $\tau$ has the density $\tau_0(t)$ given by  \eqref{eq:tau} on the open interval $]0,1[$. It follows that  $\tau=m_0\delta_0+m_1\delta_1+\tau_0(t)dt$  where $m_0, m_1\ge 0$ are possible masses at the points 0 and 1. We have further by dominated convergence
$$
m_0=\lim_{y\to 0^+}iyg(iy),\quad m_1=\lim_{y\to 0^+}iyg(-1+iy).
$$
We next evaluate the limits as $m_0=0, m_1=1$ thus showing \eqref{eq:g}.

In fact, for $y>0$
\begin{eqnarray*}
iy g(iy)&=&\left\{(1+iy)[(1+iy)(\Log(1+iy)-\Log(iy))-1]\right\}^{-1}\\
&=&\left\{(1+iy)[(1+iy)(\log\sqrt{1+y^2}+i\arctan y-\log y-i\pi/2)-1]\right\}^{-1},
\end{eqnarray*}
which tends to 0 for $y\to 0$ because of the term $-\log y$ in the denominator.

Similarly
\begin{eqnarray*}
\lefteqn{iy g(-1+iy)=\left\{(-1+iy)[iy(\Log(iy)-\Log(-1+iy))-1]\right\}^{-1}}\\
&=&\left\{(-1+iy)[iy(\log y +i\pi/2 -\log\sqrt{1+y^2}-(\pi-\arctan y))-1]\right\}^{-1},
\end{eqnarray*}
which tends to  1 for $y\to 0$.

 Note that $xg(x)\to 2$ for $x\to\infty$ because
$$
\lim_{u\to 0} g(1/u)/u=\lim_{u\to 0}\frac{u^2}{(1+u)[(1+u)\log(1+u)-u]}=2,
$$
by inserting the power series for $\log(1+u)$.  
 If this is combined with the last expression in \eqref{eq:g}, we get that $\tau_0$ is  a probability density.

The convexity and monotonicity properties of $\tau_0$ follows e.g.\ by a Maple program. \hfill $\square$

A normalized Hausdorff moment sequence is of the form
$$
\mu_n=\int_0^1 t^n\,d\mu(t),\quad n=0,1,\ldots,
$$
where $\mu$ is a probability measure on $[0,1]$. The cases where $\mu$ has the density $\tau_0(t)$ and $\tau_0(1-t)$ with respect to Lebesgue measure on the unit interval will be important in the following. 

\begin{thm}\label{thm:main2} Let $(t_n)$ denote the Hausdorff moment sequence given by
\begin{equation}\label{eq:tn}
t_n:=\int_0^1 s^n\tau_0(1-s)\,ds,\quad n\ge 0,
\end{equation}
and define
\begin{equation}\label{eq:main2}
G_\a(x):=2+\sum_{n=1}^\infty \frac{x^n}{n!}\left(t_n-\frac{\a}{n+1}\right), \quad x>0.
\end{equation}
Then $h_\a\in\mathcal{HBF}\iff G_\a(x)\ge 0,\;x>0$.
\end{thm}
\noindent {\it Proof.}
 From \eqref{eq:rho} and Proposition \ref{thm:g} it follows that 
$$
g(x)-\a\rho(x)=\int_0^\infty e^{-tx}F_\a(t)\,dt,
 $$
where
$$
F_\a(t)=e^{-t}+\int_0^1e^{-ts}\tau_0(s)ds-\a\left(\frac{1-e^{-t}}{t} - e^{-t}\right)
$$
so by \eqref{eq:hainL} and Bernstein's theorem $h_\a'\in\mathcal L$ if and only if
$F_\a(t)\ge 0$ for $t>0$.
However,
$$
e^tF_\a(t)=1+\a+\int_0^1 e^{ts}\tau_0(1-s)\,ds-\a\frac{e^t-1}{t},
$$
so inserting the power series for $e^{ts}$ and $e^t$ and using the moments $(t_n)$ from \eqref{eq:tn}, 
we get that $F_\a\ge 0$ if and only if $G_\a\geq 0$.
\hfill $\square$

\begin{rem}\label{thm:remb*}{\rm For $x>0$ we see that  $G_\a(x)$ is a decreasing function of $\a$, so if 
$G_\beta(x)\ge 0$ for $x>0$ and $\a<\beta$, then also $G_\a(x)\ge 0$ for $x>0$. This means that there is a number $\beta^*\le \a^*$ such that $h_\a'\in{\mathcal L}_0$ for $0\le \a\le \beta^*$.  Here $\a^*$ is the number from \eqref{eq:a*}. 
}
\end{rem}
\begin{rem}{\rm Defining  the sequence $(a_n)_n$ by $a_0=1, a_n=1/t_{n}-1/t_{n-1},n\ge 1$, we get
$$
t_n=1/(a_0+\cdots+a_n).
$$
Motivated by \cite[Theorem 1.1]{B:D1} we are interested in knowing if $(a_n)_n$ is a Hausdorff moment sequence. If this is true,  then $(t_n)_n$ is an infinitely divisible Hausdorff moment sequence in the sense that $(t_n^c)_n$ is a Hausdorff moment sequence for any $c>0$, cf. \cite[Proposition 4.2]{B:D2}. Numerical experiments suggest
that the necessary and sufficient conditions of Hausdorff for a sequence to be a Hausdorff moment sequence are satisfied for $(a_n)_n$, but we have not been able to prove this.

That $(t_n)_n$ is infinitely divisible is the same as the claim that the probability density on $[0,\infty[$ given by $d(s)=\tau_0(1-e^{-s})e^{-s},s\ge 0$ is infinitely divisible in the classical sense.

}
\end{rem}

\section{Two Hausdorff moment sequences}
We aim at showing the nonnegativity of the function $G_\a$ given by \eqref{eq:main2} on the positive half-line. Our approach is to show that the $n$th coefficient in the power series is positive as $n$ becomes large and to do this we shall find an asymptotic lower bound on the moments $(t_n)_n$. Not all coefficients in the series $G_\a$ are positive and it will be necessary for us to treat the first coefficients, the intermediate coefficients and the tail of the coefficients using different methods.

We first describe how to compute the moments $(t_n)_n$ of the probability measure $\tau_0(1-s)ds$ analytically, using among other things the moments of the probability measure $\tau_0(s)ds$.
We have
\begin{equation}\label{eq:pow1}
\int_0^1\frac{\tau_0(t)}{x+t}dt=\sum_{n=0}^\infty \frac{(-1)^n}{x^{n+1}}s_n,\quad x>1,
\end{equation}
where
$$%\begin{equation}\label{eq:mom}
s_n=\int_0^1 t^n\tau_0(t)dt,\quad n=0,1,\ldots.
$$%\end{equation}
Using \eqref{eq:g} we can find another expression for the power series \eqref{eq:pow1}. For this set $x=1/u$ and consider
$$
\varphi(u):=g(1/u)-\frac{1}{1+1/u}= \frac{u^3}{(1+u)[(1+u)\log(1+u)-u]}- \frac{u}{1+u}.
$$
From the power series for $\log(1+u), |u|<1$ we find
$$
(1+u)\log(1+u)-u=\sum_{n=0}^\infty \frac{(-1)^n}{(n+1)(n+2)}u^{n+2},
$$
hence
\begin{equation}\label{eq:phirel}
 \varphi(u)=\frac{2u}{1+u}\left(\sum_{n=0}^\infty\frac{(-1)^n2}{(n+1)(n+2)}u^n\right)^{-1}-\frac{u}{1+u}.
\end{equation}
Now write
$$
\left(\sum_{n=0}^\infty\frac{(-1)^n 2}{(n+1)(n+2)}u^n\right)^{-1}=\sum_{n=0}^\infty \rho_n u^n.
$$
Then $\rho_0=1, \rho_1=1/3$ and in general
$$%\begin{equation}\label{eq:pow2}
\rho_n=\sum_{k=0}^{n-1}\rho_k\frac{2(-1)^{n-1-k}}{(n-k+1)(n-k+2)},\quad n\ge 1.
$$%\end{equation}
Notice that any $\rho_n$ can be computed using this recursive formula. 
% leading to 
% $$\rho_2=-1/18,   \rho_3=7/270, \rho_4=-5/324,  \rho_5=353/34020, \rho_6=-7669/1020600.
% $$
Returning to \eqref{eq:phirel} we have
\begin{eqnarray*}
\varphi(u)&=&u\left(\sum_{n=0}^\infty (-1)^nu^n\right)\left(1+\sum_{n=1}^\infty 2\rho_n u^n\right)\\
&=& u\sum_{n=0}^\infty (-1)^n\left(1+2\sum_{k=1}^n(-1)^k\rho_k\right)u^n,
\end{eqnarray*}
which compared with \eqref{eq:pow1} shows that
$$%\begin{equation}\label{eq:mom1}
s_n=1+2\sum_{k=1}^n (-1)^k\rho_k,\quad n\ge 0.
$$%\end{equation}
The moments $(t_n)_n$ are given in terms of $(s_n)_n$ as
$$
t_n=\sum_{k=0}^n (-1)^k\binom{n}{k}s_k,\quad n=0,1,\ldots,
$$
and this makes it possible to compute first the numbers $\rho_0,\ldots,\rho_n$, then the numbers $s_0,\ldots,s_n$ and finally the numbers $t_0,\ldots,t_n$. The first 6 moments are
\begin{eqnarray*}
t_0=1,\ t_1=2/3,\ t_2=5/9,\ t_3=67/135,\ t_4=371/810,\ t_5=1465/3402.%,\\ t_6=209081/510300.
\end{eqnarray*}

Next we shall find an asymptotic lower bound on the moments of a large class of probability densities.
\begin{lemma}\label{thm:p} Let $p$ be a probability density on $]0,1[$ such that $\lim_{t\to 1}p(t)=\infty$ and with moments $\mu_n,n=0,1,\ldots$.

For any $c>0$ the exists $n_0\in\N$ such that
$$
\mu_n > \frac{c}{n+1},\quad n\ge n_0.
$$
\end{lemma}

\begin{proof} 
By assumption on $p$ there exists $0<T<1$ such that $p(t)\ge c+1$ for $T<t<1$. We then get
$$
\mu_n=\int_0^1 t^n p(t)\,dt\ge (c+1)\int_T^1 t^n\,dt=\frac{c+1}{n+1}(1-T^{n+1}). 
$$
The last expression is larger than $c/(n+1)$ iff $T^{n+1}<1/(c+1)$ which holds for 
$n$ sufficiently large.
\end{proof}

\begin{prop}\label{thm:a=2} 
 For $n\ge 4$ we have 
\begin{equation}\label{eq:in}
 t_n >  \frac{2}{(n+1)},
\end{equation}
and $G_2(x)>0$ for $x>0$. In particular $h_2\in\mathcal{HBF}$.
\end{prop}

\begin{proof} The density $\tau_0(1-t)$ is strictly increasing for $t\in [1/2,1[$ and has limit $\infty$ for $t\to 1$. Therefore there is $T\in ]1/2,1[$ such that $\tau_0(1-T)=2$ and for $T<\sigma<1$ we have
$$
t_n>\tau_0(1-\sigma)\int_{\sigma}^1 t^n\,dt=\frac{\tau_0(1-\sigma)}{n+1}(1- \sigma^{n+1}).
$$ 
The last expression is $\ge 2/(n+1)$ if and only if
$$
n+1\ge \frac{\log\left(\frac{\tau_0(1-\sigma)-2}{\tau_0(1-\sigma)}\right)}{\log(\sigma)}.
$$
Choosing $\sigma=0.985$ we get that \eqref{eq:in} holds for $n\ge 57$. 
That the inequality \eqref{eq:in} holds for $n=4,\ldots,56$ is established by a Maple program, based on the computational method described above (see Appendix \ref{A}). By \eqref{eq:main2} we therefore get for $x>0$
$$
G_2(x)>2+\sum_{n=1}^4 \frac{x^n}{n!}\left(t_n-\frac{2}{n+1}\right)=
2-\frac{x}{3}-\frac{x^2}{18}-\frac{x^3}{1620}+\frac{47}{19440}x^4>0,
$$
where the last inequality is easy to check.
\end{proof}

\begin{rem}{\rm
 As is evident from the proof above not all partial sums of the power series $G_2(x)$ are positive. The first five terms have to be combined to conclude positivity.}
\end{rem}

\noindent {\it Proof of Theorem~\ref{thm:main}.}
By Theorem~\ref{thm:main2} and Remark~\ref{thm:remb*} we know that $\b^*\le \a^*<2.3$ and  Proposition~\ref{thm:a=2} shows that $2\le\b^*$. 

The number $\b^*$ can be estimated in the following way similar to the proof of Proposition~\ref{thm:a=2}. Let $1/2<T<1$ be such that $\tau_0(1-T)=2.3$. For $T<\sigma<1$ we have
$$
t_n>\tau_0(1-\sigma)\int_{\sigma}^1 t^n\,dt=\frac{\tau_0(1-\sigma)}{n+1}(1- \sigma^{n+1}).
$$ 
The last expression is $\ge 2.3/(n+1)$ if and only if
$$
n+1\ge \frac{\log\left(\frac{\tau_0(1-\sigma)-2.3}{\tau_0(1-\sigma)}\right)}{\log(\sigma)}.
$$
Choosing $\sigma=0.989$ we get that $t_n>2.3/(n+1)$ for $n\ge 71$, while a Maple program (similar to the code in Appendix \ref{A}) shows that it holds for $n=5,\ldots,70$.
Define now the polynomials  depending on the parameter $2\le \a\le 2.3$
$$%\begin{equation}
P_N(x,\a):=2+\sum_{n=1}^N \frac{x^n}{n!}\left(t_n-\frac{\a}{n+1}\right),
$$%\end{equation}
and assume that $N\ge 5$. Since $G_\a(x)>P_N(x,\a)$ for $x>0$ we get $\a\le \b^*$ if
$P_N(x,\a)$ is non-negative for $x>0$. As an example it can be checked easily by a Maple computation that $P_N(x,\a)>0$ for $N=20$ and $\a=2.188585$.

On the other hand, since all Taylor coefficients (except the constant term) in $G_{\a}$ are less than $1$ then 
$$
G_\a(x)-P_N(x,\a)< R_N(x),\quad R_N(x)=\sum_{n=N+1}^\infty \frac{x^n}{n!},
$$ 
so if $p:=P_N(x_0,\a)$ is the global minimum of $P_N(x,\a)$ over $[0,\infty[$ assumed negative and if $R_N(x_0)<|p|$, then $G_a(x_0)<0$ and hence $\b^*<\a$.
As an example it can be checked easily by a Maple computation that $P_N(x,\a)$ for $N=20$ and $\a=2.188590$ takes the minimum value $p\approx-0.00002670$ for $x_0\approx 3.365577$ and that $R(x_0)\approx2.7\times 10^{-9}$.

Using this method $\b^*$ can easily be estimated very accurately and the computations yield the specific number in the theorem.\hfill 
$\square$ 

Let us end the paper by mentioning that $\b^*$ can be determined as the minimum value of a certain function $M$ defined in the following proposition. The graph of (approximations to) this function indicates again the minimum value as $\b^*$.
\begin{prop} The function
$$%\begin{equation}
M(x):=\frac{2+\sum_{n=1}^\infty \frac{t_n}{n!}x^n}{\sum_{n=1}^\infty \frac{x^n}{(n+1)!}},\quad x>0
$$%\end{equation}
is positive, continuous and tends to infinity for $x$ tending to  0 and to infinity.
Its minimum over $]0,\infty[$ equals $\beta^*$.
\end{prop}

\begin{proof} It is clear that $M$ tends to infinity for $x\to 0$.

 Let $c>0$ and $n_0$ be as in Lemma~\ref{thm:p} when $p(t)=\tau_0(1-t)$. Setting
$$
R_c(x)=\sum_{n=n_0}^\infty \frac{x^n}{(n+1)!},
$$
we find
$$
M(x)\ge\frac{2+\sum_{n=1}^{n_0-1}\frac{t_n}{n!}x^n+cR_c(x)}{\sum_{n=1}^{n_0-1}\frac{x^n}{(n+1)!}+R_c(x)},
$$
hence $\liminf_{x\to\infty}M(x)\ge c$. Since $c>0$ is arbitrary we get
$\lim_{x\to\infty}M(x)=\infty.$ The minimum value is clearly the largest number $\a>0$ such that $G_\a\ge 0$ on the half-line.
\end{proof}

\appendix
\section{Maple code}
\label{A}
For the reader's convenience we have included the Maple code used in the proof of Proposition \ref{thm:a=2}.
\begin{verbatim}
N:=56;

#computation of \rho_n
rho[0]:=1;
for n to N do
    rho[n] := 
    sum(2*rho[k]*(-1)^(n - 1 - k)/((n - k + 1)*(n - k + 2)), 
    k = 0 .. n - 1);
end do

#computation of s_n
s[0]:=1;
for n to N do
    s[n] := 1 + 2*sum((-1)^k*rho[k], k = 1 .. n);
end do

#computation of t_n
t[0]:=1;
for n to N do
    t[n] := sum((-1)^k*binomial(n,k)*s[k], k = 1 .. n) + s[0];
end do

#testing t_n>2/(n+1)
c:=2
min(seq(t[n]-c/(n+1), n=0..N)
\end{verbatim}

\noindent
Christian Berg\\
Department of Mathematical Sciences\\
University of Copenhagen \\
Universitetsparken 5\\
DK-2100, Denmark\\
{\em email}:\hspace{2mm}{\tt berg@math.ku.dk}

\vspace{0.5cm}

\noindent
Henrik Laurberg Pedersen\\
Department of Mathematical Sciences\\
University of Copenhagen \\
Universitetsparken 5\\
DK-2100, Denmark\\
{\em email}:\hspace{2mm}{\tt henrikp@math.ku.dk}

\end{document}